\documentclass[11pt]{article}
\usepackage{jalc}
\usepackage{bbm}
\usepackage[usenames]{color} 

\usepackage[procnames]{listings}

\usepackage{textcomp}
\usepackage{setspace}

\definecolor{white}{rgb}{1.0,1.0,1.0}
\definecolor{gray}{gray}{0.5}
\definecolor{green}{rgb}{0,0.5,0}
\definecolor{lightgreen}{rgb}{0,0.7,0}
\definecolor{purple}{rgb}{0.5,0,0.5}
\definecolor{darkred}{rgb}{0.5,0,0}
\definecolor{orange}{rgb}{1,0.5,0}
\definecolor{lbcolor}{rgb}{0.95,0.95,0.95}
\definecolor{framecolor}{rgb}{0,0.675,0.878}
\definecolor{darkblue}{rgb}{0,0,.6}

\newtheorem{theorem}{Theorem}

\newtheorem{definition}{Definition}

\newtheorem{example}{Example}

\newcommand{\set}[1]{\{#1\}}

\title{On the Number of Many-to-Many Alignments of Multiple Sequences}
\author{Steffen Eger}
\address{Computer Science Department, 
Goethe University Frankfurt am Main \\ Robert-Mayer-Stra{\ss}e 10,
Frankfurt, Germany
}
\email{steeger@em.uni-frankfurt.de}
\abstract{We count the number of alignments of $N \ge 1$ sequences when
  match-up types are from a specified set $S\subseteq \mathbb{N}^N$. Equivalently, we count
  the number of non-negative integer matrices whose rows sum to a
  given fixed vector and each of whose columns lie in $S$. We provide a
  new asymptotic formula for the case $S=\set{(s_1,\ldots,s_N) \:|\:
    1\le s_i\le 2}$.}
\keywords{Alignment, composition, sum of discrete random
  variable, lattice path}

\begin{document}
\maketitle

\section{Introduction}
Alignments of sequences arise in 
computational biology 
and in computational linguistics. In computational biology, aligning DNA sequences is a
standard task.  
In computational linguistics, aligning
(historical) variants of linguistic forms is a field of study
(see \cite{Covington:1998}). In addition, alignments of sequences
arise in computational linguistics either in 
(machine) translation, where words from different languages are
matched up, or in related string-to-string translation tasks such as
letter-to-sound conversion, where the task is to translate a
letter string into a phonetic representation, or in lemmatization, where
the task is to translate a word form into a canonical lexicon
representation. 

Traditionally, an alignment of $N$ (for an integer $N\ge 2$)
sequences of various lengths is 
defined as a manner of inserting blanks into the $N$ sequences such
that all have equal length. For example, given $\mathbf{x}=x_1$,
$\mathbf{y}=y_1y_2$ and $\mathbf{z}=z_1z_2z_3$, three (out of $239$
possible) alignments of $\mathbf{x},\mathbf{y}$ and $\mathbf{z}$ are: 

\begin{center}
\begin{tabular}{ccc}
  $x_1$ & - & -\\
  $y_1$ & $y_2$ & -\\
  $z_1$ & $z_2$ & $z_3$
\end{tabular}\hspace{0.55cm}
\begin{tabular}{cccc}
  $x_1$ & - & - & -\\
  $y_1$ & - & $y_2$ & -\\
  $z_1$ & $z_2$ & - & $z_3$
\end{tabular}\hspace{0.55cm}
\begin{tabular}{ccc}
  - & - & $x_1$\\
  $y_1$ & $y_2$ & -\\
  $z_1$ & $z_2$ & $z_3$
\end{tabular}
\quad\quad\quad $(\clubsuit)$
\end{center}
As can be seen, each such alignment has the property that, in each
position, an element 
of one of the sequences is matched up with one or zero elements from
each of the other sequences. 
In computational linguistics, alignments in which \emph{subsequences 
of length $\ge 1$} from the different sequences are matched up with
each other (`\textbf{many-to-many matches}') are oftentimes more plausible and also more frequently made
of use of (see \cite{Eger:2015,Jiam:2007}). When we allow, for example, 
in addition to the above specification, matches-up of length up to
$2$, there are several further alignments of
$\mathbf{x},\mathbf{y}$ and $\mathbf{z}$, including:
\begin{center}
\begin{tabular}{cc}
  $x_1$ & - \\
  $y_1y_2$ & -\\
  $z_1z_2$ & $z_3$
\end{tabular}\hspace{0.55cm}
\begin{tabular}{cc}
  - & $x_1$ \\
  $y_1y_2$ & -\\
  $z_1z_2$ & $z_3$
\end{tabular}\hspace{0.55cm}
\begin{tabular}{ccc}
  - & - & $x_1$\\
  - & $y_1y_2$ & -\\
  $z_1$ & $z_2$ & $z_3$
\end{tabular}
\quad\quad\quad $(\clubsuit\clubsuit)$
\end{center}

\noindent Several works have counted the number of alignments of two and more
(\textbf{`multiple'}) sequences, see
\cite{Andrade:2014,Covington:2004,Eger:2012,Eger:2013a,Griggs:1990,Rodseth:2006,Slowinski:1998,Torres:2003}. These  
works typically referred to the traditional 
definition of multiple alignments outlined above. In this note, we
count the 
number of alignments of $N$ sequences in which match-up types lie in an
arbitrary set $S$. 
The set $S$ may (but need not)
contain many-to-many matches in the above sense. Hence, our approach
generalizes the traditional setup. 

This work is structured as follows. 
Section \ref{sec:definition} gives a precise definition of multiple
alignments as we consider here. Section \ref{sec:background} places
our work in context. Section \ref{sec:counting} outlines three
theorems on multiple many-to-many alignments, which we will make use
of in Section \ref{sec:cases}: (i) a recursion, which easily allows to
calculate the number of alignments for arbitrary $S$; (ii) the multivariate
generating function 
for the number of multiple many-to-many alignments,
from which we can derive asymptotics; and (iii) a summation of the number of 
alignments 
in terms of binomial coefficients, which allows for specifying
closed-form formulas. Finally, Section \ref{sec:cases} surveys and outlines
formulas for specific $S$. In particular, we derive a new asymptotic
for the number of multiple many-to-many alignments in which subsequences
of length $1$ or $2$ are matched up with each other. 


\section{Defining multiple many-to-many alignments}\label{sec:definition}

For $N\ge 1$ fixed, let $\mathbf{x}_1,\ldots,\mathbf{x}_N$ be $N$ sequences
(over arbitrary alphabets) of finite lengths $\ell_1,\ldots,\ell_N$,
respectively.   
We define a multiple many-to-many alignment of the $N$ sequences
with respect to a set $S$ of `allowable steps'. The set $S$ defines
the valid match-up operations. 

\begin{definition}\label{def:alignment}
  Let $S$ be an arbitrary subset of $\mathbb{N}^N$, where $\mathbb{N}$
  is the set of non-negative integers. 
  We define an \emph{$S$-alignment} of
  $\mathbf{x}_1,\ldots,\mathbf{x}_N$ as any $N\times k$ matrix
  $\mathbf{A}$, for some 
  $k\ge 1$, 
   with non-negative integer entries such that
  \begin{itemize}
    \item Each row of $\mathbf{A}$ sums to the length of the respective sequence:
      $\mathbf{A}\mathbbm{1}_k=\begin{pmatrix}\ell_1\\ \vdots 
      \\ \ell_N\end{pmatrix}$. Here, $\mathbbm{1}_k$ is the vector, of
      size $k$, with all entries equal to $1$. 
    \item Each column of $\mathbf{A}$ is an element of
      $S$.\footnote{In our notation below, we do not distinguish
      column from row vectors.}
  \end{itemize}
\end{definition}
\begin{example}
  When $S=\set{(1,1),(1,2),(2,1)}$, there are seven $S$-alignments of a string of length $4$ and a string
  of length $5$. These are:
  \begin{align*}
    &\begin{pmatrix}
      2 & 1 & 1 \\
      1 & 2 & 2 \\
    \end{pmatrix}\quad
    \begin{pmatrix}
      1 & 2 & 1 \\
      2 & 1 & 2 \\
    \end{pmatrix}\quad
    \begin{pmatrix}
      1 & 1 & 2 \\
      2 & 2 & 1 \\
    \end{pmatrix}\quad
    \\
    & \begin{pmatrix}
      1 & 1 & 1 & 1 \\
      1 & 1 & 1 & 2 \\
    \end{pmatrix}\quad
    \begin{pmatrix}
      1 & 1 & 1 & 1 \\
      1 & 1 & 2 & 1 \\
    \end{pmatrix}\quad
    \begin{pmatrix}
      1 & 1 & 1 & 1 \\
      1 & 2 & 1 & 1 \\
    \end{pmatrix}\quad
    \begin{pmatrix}
      1 & 1 & 1 & 1 \\
      2 & 1 & 1 & 1 \\
    \end{pmatrix}
  \end{align*}
\end{example}
\begin{example}
  The traditional definition of multiple alignment of $N$ sequences is
  retrieved when $S=\set{(s_1,\ldots,s_N)\,|\, 0\le s_i\le
    1}-\set{(0,\ldots,0)}$. Thus, for $N=2$, $S$ has the form
  $\set{(1,0),(0,1),(1,1)}$ and for $N=3$, $S$ has the form
  \begin{align*}
    \set{(0,0,1),(0,1,0),(1,0,0),(0,1,1),(1,0,1),(1,1,0),(1,1,1)}.
  \end{align*}
  The three $S$-alignments in $(\clubsuit)$ have the matrix
  representations
   \begin{align*}\begin{pmatrix}1 & 0 & 0\\ 1 & 1 & 0 \\ 1 
    & 1 & 1\end{pmatrix}, 
    \begin{pmatrix}1 & 0 & 0 & 0 \\ 1 & 0 & 1 & 0 \\ 1 &
    1 & 0 & 1\end{pmatrix}, 
    \begin{pmatrix}0 & 0 & 1 \\ 1 & 1 & 0 \\ 1 &
    1 & 1\end{pmatrix}
    \end{align*}
  respectively. 
\end{example}
\begin{example}
  The three alignments given in $(\clubsuit\clubsuit)$ have the 
  matrix representations 
  \begin{align*}\begin{pmatrix}1 & 0 \\ 2 & 0 \\ 2 &
    1\end{pmatrix}, 
    \begin{pmatrix}0 & 1 \\ 2 & 0 \\ 2 &
    1\end{pmatrix}, 
    \begin{pmatrix}0 & 0 & 1 \\ 0 & 2 & 0 \\ 1 &
    1 & 1\end{pmatrix}
    \end{align*}
  respectively. 
  Hence, they are $S$-alignments of $\mathbf{x}, \mathbf{y},
  \mathbf{z}$ for any 
    $S\supseteq\set{(1,2,2),(0,0,1),(0,2,2),(1,0,1),(0,2,1)}$. 
\end{example}
For the remainder, we use the following \textbf{notation}. We
denote a tuple $(s_1,\ldots,s_N)$ also by $\mathbf{s}$ when this
causes no confusion. We write $\mathbf{A}_i$ for column $i$ of matrix
$\mathbf{A}$. We denote by $a_S(\ell_1,\ldots,\ell_N)$ the number of
$S$-alignments of sequences of lengths $\ell_1,\ldots,\ell_N$. We
denote by $a_S(\ell_1,\ldots,\ell_N;k)$ the number of
$S$-alignments of sequences of lengths $\ell_1,\ldots,\ell_N$ when the
number of columns ($=$ \emph{length}, or number of \emph{parts}, of the alignment) of the
respective matrices is exactly $k$. 
We let $A_{S}(\ell_1,\ldots,\ell_N)$ and $A_{S}(\ell_1,\ldots,\ell_N;k)$, respectively, denote the corresponding sets of $S$-alignments. 
We write $\mathbf{0}_N$ for the
vector $(0,\ldots,0)\in\mathbb{N}^N$.

\section{Background}\label{sec:background}
Multiple many-to-many alignments as we consider here 
generalize the concept of \emph{vector compositions} introduced in
\cite{Andrews:1975}. In fact, vector compositions are
$S$-alignments in which 
$S=\mathbb{N}^N-\set{\mathbf{0}_N}$. For the same $S$, Munarini et
al.\ \cite{Munarini:2008} introduce \emph{matrix compositions}. These
are matrices whose entries sum to a positive integer $n$ and whose columns are
non-zero. We find that the number $c^{(N)}(n)$ of matrix compositions 
with $N$ rows satisfies
$c^{(N)}(n)=\sum_{\ell_1+\cdots+\ell_N=n}a_S(\ell_1,\ldots,\ell_N)$. 

Multiple many-to-many alignments are also closely
related to \emph{lattice path combinatorics} 
(see \cite{Stanley:2012}). Lattice paths 
are paths from the origin $\mathbf{0}_N$ to some point 
$(\ell_1,\ldots,\ell_N)$ in which each step lies in some set $S$. In our
case, each coordinate of each step $\mathbf{s}\in S$ is non-negative.   

When $S$ is fixed and finite, then the class 
\begin{align*}
  L_S=\set{\mathbf{A}\in\mathbb{N}^{N\times k}\,|\,\mathbf{A} \text{ satisfies
    Definition \ref{def:alignment} for some
  } (\ell_1, \ldots, \ell_N) \text{ and
        some $k$}}   
\end{align*}
is clearly a \emph{regular language}. 
In fact, 
$L_S$ is given by the regular expression
$(\mathbf{s}_1|\cdots|\mathbf{s}_r)^+$, where 
$\mathbf{s}_1,\ldots,\mathbf{s}_r$ are the elements of $S$.  
Munarini et al.\ \cite{Munarini:2008} specify an encoding
showing that $L_S$ is regular even when
$S=\mathbb{N}^N-\set{\mathbf{0}_N}$. 

Multiple many-to-many alignments 
are also related to \emph{sums of discrete multivariate random
  variables}. In fact, let $X_1,\ldots,X_k$ be independently and
identically distributed random vectors, each of whose range is
$S$. Then,
\begin{align*}
  P[X_1+\cdots+X_k=(\ell_1, \ldots, \ell_N)] = \sum_{\mathbf{A}\in 
    A_S(\ell_1,\ldots,\ell_N;k)}P[\mathbf{A}_1]\cdots P[\mathbf{A}_k].
\end{align*} 
When each $X_i$ is uniformly distributed on $S$, then
\begin{align*}
  P[X_1+\cdots+X_k=(\ell_1, \ldots, \ell_N)] = \frac{1}{|S|^k}a_S(\ell_1,\ldots,\ell_N;k). 
\end{align*} 
Since the sum $X_1+\cdots+X_k$ is asymptotically normal, the last
equality allows to approximate $a_S(\ell_1,\ldots,\ell_N;k)$ by a
normal curve in the spirit of \cite{Eger:2014}. 

Finally, it is well-known that alignment problems are closely related to
the edit distance problem \cite{Levenshtein:1966}. Edit distance
measures the minimal 
number of insertions, deletions, and substitutions necessary to
transform one string into another. Allowing steps in a
set $S$ would correspond to 
allowing transformation patterns as specified by $S$, such as
insertion (via the step $(0,1)$), deletion ($(1,0)$), or more complex
transformations such as inversion ($(2,2)$), expansion ($(1,2)$) or
squashing ($(2,1)$) (see \cite{Eger:2013a,Oommen:1995}). Considering
$N\ge 
2$ sequences further generalizes this setup.

\section{Counting the number of multiple many-to-many alignments}\label{sec:counting}

Theorem
\ref{theorem:count} outlines a recurrence which
$a_S(\ell_1,\ldots,\ell_N)$ satisfies and which allows its efficient
evaluation. 

\begin{theorem}\label{theorem:count}
  The quantity $a_S(\ell_1,\ldots,\ell_N)$ satisfies the recurrence 
  \begin{align*}
    a_S(\ell_1,\ldots,\ell_N) = \sum_{(s_1,\ldots,s_N)\in S}
    a_S(\ell_1-s_1,\ldots,\ell_N-s_N) 
  \end{align*}
  with initial conditions $a_S(0,\ldots,0)=1$ and $a_S(n_1,\ldots,n_N)=0$
  whenever $n_m<0$ for some $1\le m\le N$. 
\end{theorem}
\begin{proof}
  Each $S$-alignment of some length $k$ is made up of a last column
  $\mathbf{s}=(s_1,\ldots,s_N)\in S$ and an arbitrary $S$-alignment of
  length $k-1$ of $(\ell_1-s_1,\ldots,\ell_N-s_N)$.
\end{proof}

Below we show sample Python\footnote{See \url{www.python.org}.} code that
computes $a_S(\ell_1,\ldots,\ell_N)$ efficiently based on
Theorem \ref{theorem:count}.  
\begin{lstlisting}
   import itertools,numpy as np

   def a(l,S):
     # l = [l_1,...,l_N]
     N = len(l)
     indices = [range(l_i+1) for l_i in l]
     zero = tuple([0 for i in xrange(N)])
     table = {}
     for multiindex in itertools.product(*indices):
        multiindex = tuple(multiindex)
        if multiindex==zero:
            table[multiindex]=1
        else:
            local = 0
            for s in S:
                index = tuple(np.array(multiindex)-np.array(s))
                local += table.get(index,0)
            table[multiindex] = local
    return table[tuple(l)]
\end{lstlisting}

\begin{theorem}
  Let $f(z_1,\ldots,z_N)=\sum_{n_1,\ldots,n_m\ge 0}
  a_S(n_1,\ldots,n_m)z_1^{n_1}\cdots z_N^{n_m}$ be the multivariate
  generating function for the number of 
  $S$-alignments. Then $f(z_1,\ldots,z_N)$ has the representation
  \begin{align*}
    f(z_1,\ldots,z_N) = \frac{1}{1-\sum_{(s_1,\ldots,s_N)\in S}\limits
      z_1^{s_1}\cdots z_N^{s_N}}. 
  \end{align*}
\end{theorem}
\begin{proof}
  The generating function for $a_S(\ell_1,\ldots,\ell_N;k)$ is easily
  seen to be
  \begin{align*}
    f(z_1,\ldots,z_N;k) = \Bigl(\sum_{\mathbf{s}\in S} z_1^{s_1}\cdots
    z_N^{s_N}\Bigr)^k. 
  \end{align*}
  The result then follows by summing over $k$. 
\end{proof}

\begin{theorem}\label{theorem:3}
  Let the elements in $S$ be enumerated as
  $\mathbf{s}_1,\mathbf{s}_2,\ldots$. 
  Moreover, let $k\ge 0$ and let
  \begin{align*}
    B_S(k,\ell_1,\ldots,\ell_N)=\set{(r_1,r_2,\ldots,r_t)\,|\, 
    r_i\ge 0,\,\sum_{i=1}^t r_i = k,\, \sum_{i=1}^t\mathbf{s}_ir_i =
    (\ell_1,\ldots,\ell_N)}. 
  \end{align*}
  Here, $t$ is the size of $S$.\footnote{When $S$ has infinite
    size, then all $r_i$ must be zero for which $\mathbf{s}_i$ has a
    component that exceeds the corresponding component of
    $(\ell_1,\ldots,\ell_N)$; hence, only a finite and fixed number $t'$
    of indices in 
    $(r_1,r_2,\ldots,r_t)$ can be non-zero in this case, too.}
  Then 
  \begin{align}\label{eq:rewrite}
    a_S(\ell_1,\ldots,\ell_N) = \sum_{k\ge
      0} \sum_{(r_1,r_2,\ldots,r_t)\in B_S(k,\ell_1,\ldots,\ell_N)}\binom{k}{r_1,r_2,\ldots,r_t},
  \end{align}
  where $\binom{k}{r_1,r_2,\ldots,r_t}=\frac{k!}{r_1!r_2!\cdots r_t!}$ are the
  \emph{multinomial coefficients}. 
\end{theorem}

\begin{proof}
  We first note that $a_S(\ell_1,\ldots,\ell_N)=\sum_{k\ge
    0}a_S(\ell_1,\ldots,\ell_N;k)$. We then find the above formula for
  $a_S(\ell_1,\ldots,\ell_N;k)$ by matching up columns of the same
  type $\mathbf{s}_i$; $r_i$ is the \emph{multiplicity} of column type
  $\mathbf{s}_i$. 
\end{proof}

\section{Special cases}\label{sec:cases}
\subsection{The case $N=1$}\label{sec:deg}
The case $N=1$, which might be considered a degenerate case of an
alignment, yields the number of \emph{$S$-restricted integer
  compositions} (see \cite{Eger:2013,Heubach:2004}) of $\ell$, i.e., the
number 
of solutions  
$(\pi_1,\ldots,\pi_k)$ such that $\pi_1+\cdots+\pi_k=\ell$, and where
each $\pi_i\in S$. Depending on $S$, there are several closed-form
solutions for $a_S(\ell)$ as exemplified in Table
\ref{table:compositions}.  

\begin{table}[!htb]
  \begin{center}
  \begin{tabular}{|ll|}\hline
    $S=\set{1,2,3,\ldots}$ & $a_S(\ell) = 2^{\ell-1}$ \\
    $S=\set{1,2}$ & $a_S(\ell)=F_{\ell+1}$ \\
    $S=\set{2,3,4,\ldots}$ & $a_S(\ell)=F_{\ell-1}$ \\
    $S=\set{1,3,5,7,\ldots}$ & $a_S(\ell)=F_{\ell}$ \\
    $S=\set{1,2,\ldots,M}$ & $a_S(\ell)\sim
    \frac{\phi^{{\ell+1}}}{G'(\sigma)}$
    \\
    \hline
  \end{tabular}
  \caption{$a_S(\ell)$ for different
    $S\subseteq\mathbb{N}-\set{\mathbf{0}_1}$. Here, $F_n$ denotes the
    $n$-th Fibonacci number (see, e.g., 
\cite{Gessel:2013, Shapcott:2013}). Moreover, $M\ge 1$ and $\phi$ is
the unique positive real solution to
$\frac{1}{X^1}+\cdots+\frac{1}{X^M}=1$, $G(x)=x^1+\cdots+x^M$ and $G'$
denotes the first derivative of $G$, and $\sigma=1/\phi$
(see \cite{Malandro:2011}).}
  \label{table:compositions}
  \end{center}
\end{table}

\subsection{The case $N=2$}
\subsubsection{The case $S=\set{(1,0),(0,1),(1,1)}$}
The case $S=\set{(1,0),(0,1),(1,1)}$ is the classical case of
alignments of exactly two sequences with 
(simple) matches and skips (see
\cite{Andrade:2014,Covington:2004}). It leads to the  
well-known \emph{Delannoy numbers} (see \cite{Banderier:2004}). The central
Delannoy numbers are listed as integer sequence A001850
(see \cite{Sloane}). Table \ref{table:delannoy} provides two
identities for $a_S(\ell_1,\ell_2)$ and one 
approximate formula for the numbers. We note that the second formula
immediately follows from Theorem \ref{theorem:3} in a similar manner as
for $a_{\set{(1,1),(1,2),(2,1)}}(\ell_1,\ell_2)$ 
in Section \ref{sec:next}. 
\begin{table}[!htb]
  \begin{center}
  \begin{tabular}{|cc|} \hline
    Exact & $\sum_{d\ge 0}2^d\binom{\ell_1}{d}\binom{\ell_2}{d}$ \\
          & $\sum_{d\ge
      0}\frac{(\ell_1+\ell_2-d)!}{d!(\ell_1-d)!(\ell_2-d)!}$ \\
    Approximate & $(\frac{r+\ell_2}{\ell_1})^{\ell_1}(\frac{r+\ell_1}{\ell_2})^{\ell_2}$\\
    \hline
  \end{tabular}
  \caption{Exact formulas and approximations for Delannoy
    numbers. Here $r=\sqrt{\ell_1^2+\ell_2^2}$
    (see \cite{Kiselman:preprint,Kiselman:2008}).}
  \label{table:delannoy}
  \end{center}
\end{table}
\subsubsection{The case $S=\set{(1,1),(1,2),(2,1)}$}\label{sec:next}
The case $S=\set{(1,1),(1,2),(2,1)}$ is given as integer sequence
A191588. Its diagonal $a_S(\ell,\ell)$ is given as sequence
A098479. Table \ref{table:1-2} lists two formulas, whose proofs we
sketch 
below. A more general formula for the case
$\set{(1,1),(1,2),\ldots,(1,A)}\cup\set{(2,1),\ldots,(B,1)}$, for
$A,B\ge 2$, is
provided in \cite{Eger:2013a}. 
\begin{table}[!htb]
  \begin{center}
  \begin{tabular}{|cc|} \hline
    ($\star$) &
    $\sum_{k=0}^{\ell_1}\binom{k}{2k-\ell_1}\binom{2k-\ell_1}{\ell_2-k}$\\
    ($\star\star$) &
    $\sum_{k=\lceil\frac{\max\set{\ell_1,\ell_2}}{2}\rceil}^{\min\set{\ell_1,\ell_2}}\limits\sum_{j\ge 0}(-1)^j\binom{k}{j}\binom{k-j}{\ell_1-k-j}\binom{k-j}{\ell_2-k-j}$\\
    \hline
  \end{tabular}
  \caption{Closed-form formulas for
    $a_{\set{(1,1),(1,2),(2,1)}}(\ell_1,\ell_2)$.} 
  \label{table:1-2}
  \end{center}
\end{table}
\begin{proof}
Proof of formula ($\star$): 
Let
${S}=\set{(1,1),(1,2),(2,1)}$. The
constraints on the multiplicities $(r_1,r_2,r_3)$ in
Eq.~\eqref{eq:rewrite}, Theorem \ref{theorem:3}, can then be rewritten as 
\begin{align*}
  \begin{pmatrix}
    r_1\\r_2\\r_3
  \end{pmatrix} = 
  \begin{pmatrix}
    1 & 1 & 1 \\
    1 & 1 & 2 \\ 
    1 & 2 & 1 
  \end{pmatrix}^{-1}
  \begin{pmatrix}k \\ \ell_1 \\ \ell_2 \end{pmatrix} = 
  \begin{pmatrix}
    3 & -1 & -1 \\
    -1 & 0 & 1 \\
    -1 & 1 & 0 
  \end{pmatrix}
  \begin{pmatrix}k \\ \ell_1 \\ \ell_2 \end{pmatrix}
\end{align*}
which leads to the formula $\sum_{k\ge
  0}\frac{k!}{(\ell_1-k)!(\ell_2-k)!(3k-\ell_1-\ell_2)!}$, which is
equivalent to the formula indicated. 
\end{proof}
\begin{proof}
  Proof of formula ($\star\star$):
  This identity follows by applying the inclusion/exclusion principle
  to the sets $A_i = $ set of $2\times k$ matrices  with entries
  in $\set{1,2}$ whose first and
  second row sum to $\ell_1$ and $\ell_2$, respectively, 
  such that 
  column $i$ consists of $2$'s. For fixed number 
  $k$ of columns, we note that
  $a_S(\ell_1,\ell_2;k)=|A_1^c\cap\cdots\cap A_k^c|$. 
\end{proof}

\subsubsection{The case $S=\set{(1,1),(1,2),(2,1),(2,2)}$}
The diagonals $a_S(\ell,\ell)$ of this case are listed as integer
sequence A051286.  
They are also known as \emph{Whitney numbers} of the fence poset
$A_\ell$, see \cite{Bona:2010}. Table \ref{table:whitney} below lists an exact formula
for $a_S(\ell_1,\ell_2)$, which
can be derived as sketched in Section \ref{sec:independence}, and an
approximate formula for $a_S(\ell,\ell)$ which is provided in the comments for the
respective integer sequence. 
\begin{table}[!htb]
  \begin{center}
  \begin{tabular}{|cc|} \hline
    Exact & $\sum_{k\ge 0} \binom{\ell_1-k}{k}\binom{\ell_2-k}{k}$\\
    Approximate & $\sqrt{3+\frac{7}{\sqrt{5}}} \cdot \frac{(\frac{1+\sqrt{5}}{2})^{2\ell}}{\sqrt{8\pi \ell}}$ \\
    \hline
  \end{tabular}
  \caption{Exact and approximate formulas for
    $a_{\set{(1,1),(1,2),(2,1),(2,2)}}(\ell_1,\ell_2)$ and
    $a_{\set{(1,1),(1,2),(2,1),(2,2)}}(\ell,\ell)$, respectively.}
  \label{table:whitney}
  \end{center}
\end{table}

\subsubsection{The case $S=\set{(x,y)\,|\, x\ge 1,y\ge 0}$}
The diagonal $a_S(\ell,\ell)$ of this case is integer sequence
A047781. A closed-form 
formula for $a_S(\ell_1,\ell_2)$ can easily be established as (see
\cite{Kimberling:2001}) 
\begin{align*}
  \sum_{k\ge 0}\binom{\ell_1-1}{k-1}\binom{\ell_2+k-1}{k-1}.
\end{align*}
This result follows from Section \ref{sec:independence} and by noting
that the number of \emph{compositions}
(that is, $\set{1,2,3,\ldots}$-restricted integer compositions in the sense of
Section \ref{sec:deg}) of $n$ with exactly $k$ parts is
$c_{\set{1,2,3,\ldots}}(n,k)=\binom{n-1}{k-1}$ and the number of \emph{weak
compositions} ($\mathbb{N}$-restricted integer compositions in the
above sense) of $n$ with exactly $k$ parts is 
$c_{{\set{0,1,2,3,\ldots}}}(n,k)=\binom{n+k-1}{k-1}$. A comment for
the integer sequence lists the asymptotic
\begin{align*}
  a_S(\ell,\ell)\sim
  \frac{2^{1/4}\cdot(3+2\sqrt{2})^\ell}{4\sqrt{\pi\ell}}. 
\end{align*}

\subsubsection{Similar cases}
We may construct more such examples \emph{ad libitum}. For instance, the
case $S=\set{(1,1),(1,3),(3,1)}$ is listed as integer sequence
A098482. The case $S=\set{(1,0),(1,1),(1,2),(2,1)}$ is listed as
integer sequence A191354. When $S=\set{(x,y)\,|\, 1\le x\le 3, 1\le
  y\le 3}$, then similarly as for the Whitney numbers, we may derive 
\begin{align*}
  a_S(\ell_1,\ell_2)=\sum_{k\ge 0}\left(\sum_{i\ge
    0}\binom{k}{i}\binom{k-i}{\ell_1-k-2i}\right)\left(\sum_{i\ge
    0}\binom{k}{i}\binom{k-i}{\ell_2-k-2i}\right). 
\end{align*}

\subsection{The case $N=3$}
Let $S=\set{(x,y,z)\,|\, 0\le x,y,z\le 1}-\set{\mathbf{0}_3}$. Then, 
$a_S(\ell,\ell,\ell)$ is integer sequence A126086. The comments for
this sequence list 
%
as approximate formula: 
\begin{align*}
  a_S(\ell,\ldots,\ell) \sim c\cdot d^\ell, 
\end{align*}
where $d=12\cdot 2^{2/3}+15\cdot 2^{1/3}+19$ and $c$ is a specified
constant. Further formulas can be retrieved as special cases of the
results in Section \ref{sec:bio}.

\subsection{The case $S=S_{X_1}\times\cdots\times S_{X_N}$}\label{sec:independence} 
If $S_{X_1},\ldots,S_{X_N}\subseteq \mathbb{N}$ are base sets and $S=S_{X_1}\times\cdots\times S_{X_N}$ (`independence'), then
$a_S(\ell_1,\ldots,\ell_N)$ is given by
\begin{align*}
  \sum_{k\ge 0}c_{S_{X_1}}(\ell_1;k)\cdots c_{S_{X_N}}(\ell_N;k),
\end{align*}
where $c_{S_X}(\ell;k)$ denotes the number of composition of $\ell$
with exactly $k$ parts, each of which is in $S_X$. 
Accordingly, when $\ell_1=\cdots=\ell_N=\ell$ and all base sets are identical to some set $S_X$, then this becomes 
\begin{align*}
  a_S(\ell,\ldots,\ell)=\sum_{k\ge 0}c_{S_X}(\ell;k)^N.
\end{align*}
When, e.g., $S_X=\set{1,\ldots,M}$, for some $M\ge 1$, then
the numbers $c_{S_X}(\ell;k)$ are unimodal in $\ell$ for fixed $k$
(see \cite{Belbachir:2007,Belbachir:Farid:2007,Belbachir:2008}) and may be
approximated by a 
normal curve as in \cite{Eger:2014}.  

\subsection{The case $S=\set{(s_1,\ldots,s_N) \:|\:    0\le s_i\le
    1}-\set{\mathbf{0}_N}$}\label{sec:bio}
This is the classical case of alignments of $N$ sequences considered
in computational biology. An exact formula for
$a_S(\ell_1,\ldots,\ell_N)$ is given by (see
\cite{Slowinski:1998}) 
\begin{align}\label{eq:Slowinski}
  \sum_{k=\max\set{\ell_1,\ldots,\ell_N}}^{\ell_1+\cdots+\ell_N}\sum_{j=0}^k (-1)^j\binom{k}{j}\prod_{j=1}^N\binom{k-j}{\ell_j}.
\end{align} 
Eq.~\eqref{eq:Slowinski} can be
derived via the inclusion/exclusion principle very similarly as
above in Section \ref{sec:next}. 
Another closed-form formula is given in a comment to integer sequence
A126086, namely:
\begin{align}\label{eq:N=3}
  a_S(\ell_1,\ldots,\ell_N) = \sum_{\mu\ge
    0}\binom{\mu}{\ell_1}\cdots\binom{\mu}{\ell_N}\cdot\frac{1}{2^{\mu+1}}.
\end{align}
An approximate formula for $a_S(\ell,\ldots,\ell)$ has also been 
established (see \cite{Griggs:1990,Raichev:2007}):
\begin{align*}
  a_S(\ell,\ldots,\ell)\sim
  (2^{1/N}-1)^{-N\ell}\frac{1}{(2^{1/N}-1)2^{(N^2-1)/2N}\sqrt{N(\pi\ell)^{N-1}}}. 
\end{align*}

\subsection{The case $S=\mathbb{N}^N-\set{\mathbf{0}_N}$}
In the case when no non-zero step from $\mathbb{N}^N$ is discarded, a
formula is given by (see \cite{Andrews:1975}):
\begin{align*}
  a_S(\ell_1,\ldots,\ell_N) = \sum_{k\ge 0}^{\ell_1+\cdots+\ell_N}\sum_{i=0}^k
  (-1)^i\binom{k}{i}\prod_{j=1}^N\binom{\ell_j+k-i-1}{\ell_j}. 
\end{align*}
It has been noted (see \cite{Duchi:2004}) that
$a_{\mathbb{N}^N-\set{\mathbf{0}_N}}(\ell,\ldots,\ell)=2^{\ell-1}a_{\set{(s_1,\ldots,s_N)
    \:|\:    0\le s_i\le 
    1}-\set{\mathbf{0}_N}}(\ell,\ldots,\ell)$, which leads to the formula
\begin{align*}
  a_{\mathbb{N}^N-\set{\mathbf{0}_N}}(\ell,\ldots,\ell) = \sum_{\mu\ge 0} \binom{\mu}{\ell}^N2^{\ell-\mu-2}.
\end{align*} 

\subsection{The case $S=\set{(s_1,\ldots,s_N) \:|\:    1\le s_i\le 2}$}\label{sec:new}
As a closed-form formula for this case, we obtain as a 
specialization 
of our above results in Section \ref{sec:independence}:
\begin{align*}
  a_S(\ell_1,\ldots,\ell_N) = \sum_{k\ge 0}\prod_{i=1}^N\binom{\ell_i-k}{k}.
\end{align*}
To derive an asymptotic formula for $a_S(\ell,\ldots,\ell)$, we resort
to the techniques for computing asymptotics of diagonal coefficients
of multivariate generating 
functions (see \cite{Pemantle:2008,Raichev:2007}).
This leads us to the following theorem.
\begin{theorem}\label{theorem:main}
  Let $\phi=\frac{\sqrt{5}-1}{2}$. Moreover, let
  $A=-\phi^{N-1}(1+\phi)^{N-1}(1+2\phi)$. 
  Define $h
  =N\Bigl(\frac{\phi}{1+3\phi+2\phi^2}\Bigr)^{N-1}$ and 
  $b_0=\frac{1}{-\phi A \sqrt{(2\pi)^{N-1}h}}$. Then
  \begin{align*}
    a_S(\ell,\ldots,\ell) \sim \phi^{-\ell N}b_0{\ell}^{(1-N)/2}.
  \end{align*}
\end{theorem}
\begin{proof}
  Let
  $\frac{I(\mathbf{z})}{J(\mathbf{z})}=\frac{1}{1-\sum_{\mathbf{s}\in
      S}\mathbf{z}^{\mathbf{s}}}$ be the multivariate generating function
  for the number of $S$-alignments, with $S$ as indicated. With
  notation as in \cite{Raichev:2007}, we call
  $\mathtt{CRIT}(\ell,\ldots,\ell)$ the set of solutions $\mathbf{x}\in\mathbb{R}^N$ satisfying
  $J(\mathbf{x})=0$ and $\ell x_i\frac{\partial J(\mathbf{x})}{\partial 
    z_i}=\ell x_N\frac{\partial J(\mathbf{x})}{\partial z_N}$ (for all 
  $i<N$). Then, $\mathtt{CRIT}(\ell,\ldots,\ell)$ is finite and there
  is exactly one  \emph{contributing
    point}
  $\mathbf{c}=(\phi,\ldots,\phi)$, as defined in 
  \cite{Raichev:2007}. We
  find $\phi$ by induction on $N$, starting from $N=1$: $1-x-x^2=0$;
  in general, we consider the positive real root of
  $1-\sum_{i=0}^N\binom{N}{i}x^{N+i}$. Finally, we apply Proposition
  3.4 and Theorems 3.1-3.3 in \cite{Raichev:2007}
  to derive the asymptotic.  
\end{proof}
\begin{example}\label{example:comparison}
Table \ref{table:comparison} compares the numbers
$a_S(\ell,\ell,\ell)$ for $S_0=\set{(s_1,s_2,s_3) \:|\:    1\le
  s_i\le 2}$ and $S_1=\set{(s_1,s_2,s_3) \:|\:    0\le s_i\le
    1}-\set{\mathbf{0}_3}$, respectively. We note that the latter numbers
grow much more rapidly. In fact, under $S=S_0$, $a_S(10,10,10)$ is
still quite moderate (68933), while under $S=S_1$, this number 
amounts to almost $10^{16}$. 
  \begin{table}[!htb]
    \begin{center}
    \begin{tabular}{|lll|} \hline
       $\ell$      & $\set{(s_1,s_2,s_3) \:|\:    1\le s_i\le 2}$ & $\set{(s_1,s_2,s_3) \:|\:    0\le s_i\le
    1}-\set{\mathbf{0}_3}$\\ \hline
      1 &1& 13 \\
2 &2& 409\\
3 &9& 16081\\
4 &29& 699121\\
5 &92& 32193253\\
6 &343 & 1538743249\\
7 &1281 & 75494983297\\
8 &4720 & 3776339263873\\
9 &17899 & 191731486403293\\
10 &68933 & 9850349744182729\\
      \hline
    \end{tabular}
    \caption{Comparing numbers for cases of Sections \ref{sec:bio} and
      \ref{sec:new}. Numbers are calculated via Theorem
\ref{theorem:count}.}
    \label{table:comparison}
    \end{center}
  \end{table}
\end{example}
\begin{example}
  We let $N=3$, and
  $S=\set{(1,1,1),(2,1,1),(1,2,1),(1,1,2),(2,2,1),(2,1,2),(1,2,2),(2,2,2)}$. Table
  \ref{table:sketch} sketches exact and approximate values, computed
  from Theorem \ref{theorem:main}, of
  $a_S(\ell,\ell,\ell)$ for $\ell=1,\ldots,20$, and approximation
  errors. 
  \begin{table}[!htb]
    \begin{center}
    \begin{tabular}{|llll|} \hline
             &  (I) & (II) & \\
      $\ell$ & $a_S(\ell,\ell,\ell)=$ & $a_S(\ell,\ell,\ell)\sim$ &
      $\frac{|{\text{(I)}-\text{(II)}}|}{\text{(I)}}$ \\ \hline
      1 &1& 1.6489 & 0.393\\
2 &2& 3.4924 & 0.427 \\
3 &9& 9.8626 & 0.087\\
4 &29& 31.334 & 0.074\\
5 &92& 106.19 & 0.133 \\
6 &343& 374.84 & 0.084\\
7 &1281& 1361.00 & 0.058 \\
8 &4720& 5044.70 & 0.064 \\
9 &17899& 18995.30 & 0.057 \\
10 &68933& 72418.85 & 0.048\\
11 &266364& 278882.88 & 0.044 \\
12 &1037423& 1082919.63 & 0.042\\
13 &4072439& 4234450.32 & 0.038 \\
14 &16065148& 16656175.18 & 0.035 \\
15 &63658521& 65852910.95 & 0.033 \\
16 &253356763& 261522569.36 & 0.031 \\
17 &1012049086& 1042661064.91 & 0.029 \\
18 &4055596343& 4171406306.87 & 0.027 \\
19 &16299779331& 16740341694.65 & 0.026 \\
20 &65683233938& 67367564115.86 & 0.025 \\
      \hline
    \end{tabular}
    \caption{Exact and approximate values of $a_S(\ell,\ell,\ell)$ for $S=\set{(s_1,s_2,s_3) \:|\:    1\le
  s_i\le 2}$. Exact numbers are calculated via Theorem
\ref{theorem:count}.}
    \label{table:sketch}
    \end{center}
  \end{table}
\end{example}

\bibliographystyle{jalc}
\bibliography{lit}

\end{document}